\documentclass[12pt]{amsart}

\usepackage{graphicx}
\usepackage{verbatim}
\usepackage{amsmath,amsthm,amsfonts}
\usepackage{amssymb}
\usepackage{times}
\usepackage{enumitem}
\usepackage[all]{xy}
\usepackage{esvect}

\newtheorem{theorem}{Theorem}
\newtheorem{proposition}[theorem]{Proposition}

\newtheorem{lemma}[theorem]{Lemma}

\newtheorem{corollary}[theorem]{Corollary}

\theoremstyle{definition}

\newtheorem{def-theorem}[theorem]{Definition-Theorem}
\newtheorem{remark}[theorem]{Remark}

\newtheorem{question}[theorem]{Question}

\newcommand{\be}{\begin{equation}}
\newcommand{\ee}{\end{equation}}
\newcommand{\bea}{\begin{eqnarray}}
\newcommand{\eea}{\end{eqnarray}}
\newcommand{\beas}{\begin{eqnarray*}}
\newcommand{\eeas}{\end{eqnarray*}}
\newcommand{\ba}{\begin{array}}
\newcommand{\ea}{\end{array}}

\newcommand{\bbG} {\mathbb{G}}		

\newcommand{\bbA} {\mathbb{A}}

\newcommand\INTO{\ar@{^{(}->}[r]}

\newcommand{\GL}{\operatorname{GL}}

\newcommand{\PGL}{\operatorname{PGL}}

\newcommand{\id}{\operatorname{id}}

\newcommand{\Sym}{\operatorname{\Sigma}}

\newcommand{\Spec}{\operatorname{Spec}}

\newcommand{\cO}{\mathcal{O}}

\newcommand{\presectionspace}{\vspace{0.2cm}} 

\begin{document}

\author{Zinovy Reichstein}
\address{Department of Mathematics\\University of British Columbia\\ BC, Canada V6T 1Z2}
\email{reichst@math.ubc.ca}
\thanks{Partially supported by
National Sciences and Engineering Research Council of
Canada Discovery grant 253424-2017.}

\title[Fields of cross-ratios]{On a rationality problem for fields of cross-ratios}



\begin{abstract}
 Let $k$ be a field, $n \geqslant 5$ be an integer, 
$x_1, \dots, x_n$ be independent variables and
$L_n = k(x_1, \dots, x_n)$.
The symmetric group $\Sym_n$ acts on $L_n$ by permuting the variables, and
the projective linear group $\PGL_2$ acts by applying (the same) fractional linear transformation 
to each varaible. The fixed field $K_n = L_n^{\PGL_2}$ is called ``the field of cross-ratios". Let $S \subset \Sym_n$
be a subgroup. The Noether Problem asks whether the field extension $L_n^S/k$ is rational, and the Noether Problem 
for cross-ratios asks whether $K_n^S/k$ is rational. In an effort to relate these two problems, 
H.~Tsunogai posed the following question: Is $L_n^S$ rational over $K_n^S$? He answered this question 
in several situations, in particular, in the case where $S = \Sym_n$. In this paper we extend his results
by recasting the problem in terms of Galois cohomology. Our main theorem asserts that the following conditions on 
a subgroup $S \subset \Sym_n$ are equivalent: (a) $L_n^S$ 
is rational over $K_n^S$, (b) $L_n^S$ is unirational over $K_n^S$, (c) $S$ has an orbit of odd order 
in $\{ 1, \dots, n \}$. 
\end{abstract}

\maketitle

\presectionspace
\section{Introduction}

Let $k$ be a base field, $n \geqslant 5$ be an integer, 
$x_1, \dots, x_n$ be independent variables, and
\[ L_n = k(x_1, \dots, x_n). \] 
The group $\PGL_2$ acts on $L_n$ 
via
\[ \begin{pmatrix} a & b \\ c & d \end{pmatrix} \cdot x_i \to 
\frac{a x_i + b}{c x_i + d} \]
for $i = 1, \dots, n$. The field of invariants 
$K_n = L_n^{\PGL_2}$ is 
generated over $k$ by the $n-3$ cross-ratios 
\[ [x_1, x_2, x_3, x_i] = \text{$\dfrac{(x_i - x_1)(x_3 - x_2)}{(x_i - x_2)(x_3 - x_1)}$ , \quad 
 $i = 4, \dots, n$.} \] 
For this reason we will refer to $K_n$ as the field of cross-ratios. 
The natural action of the symmetric group $\Sym_n$ on $L_n$ induced 
by permuting the variables descends to a faithful action on $K_n$. 
Suppose $S$ is a subgroup of $\Sym_n$. 

The Noether problem asks whether the fixed field $L_n^S$ is rational 
(respectively, stably rational or retract rational) over $k$.  The Noether Problem 
for cross-ratios is whether or not $K_n^S$ is rational (respectively, stably rational or retract 
rational) over $k$. In an effort to relate these two problems, H.~Tsunogai~\cite{tsunogai} 
posed the following question:

\begin{question} \label{q.main}
Is $L_n^S$ is rational over $K_n^S$?
\end{question}

He answered this question in several situations (see~\cite[Theorems 1, 2, 3]{tsunogai})
in particular, in the case, where $S = \Sym_n$.
Our main theorem generalizes his results as follows. 

\begin{theorem} \label{thm.main} 
Let $S$ be a subgroup of the symmetric group $\Sym_n$, where $n \geqslant 5$.
Then the following conditions are equivalent:

\begin{enumerate}
    \item[\rm{(a)}] $L_n^S$ is rational over $K_n^S$,
     \item[\rm{(a)}] $L_n^S$ is unirational over $K_n^S$,
     \item[\rm{(c)}] $S$ has an orbit of odd order in $\{ 1, \dots, n \}$.
\end{enumerate}
\end{theorem}

The remainder of this note will be devoted to proving Theorem~\ref{thm.main}.

\section{Recasting the problem in the language of Galois cohomology}
\label{sect.recast}

Let $G$ be the subgroup of $(\GL_2)^n = \GL_2 \times \dots \times \GL_2$ consisting of $n$-tuples $(g_1, \dots, g_n)$ 
such that $\overline{g_1} = \dots = \overline{g_n}$. Here $\overline{g}$ denotes the image of $g \in \GL_2$ in $\PGL_2$.
In other words, $(g_1, \dots, g_n) \in (\GL_2)^n$ lies in $G$ if and only if $g_1, \dots, g_n$ are scalar multiples of each other.
The symmetric group $\Sym_n$ acts naturally on $(\GL_2)^n$ by permuting the entries; $G$ is invariant under this action.
For any subgroup $S \subset \Sym_n$, we will denote the semidirect product $G \rtimes S$ by $G_S$. This gives rise to 
the natural split exact sequence 
\begin{equation} \label{e.seq1} \xymatrix{  1 \ar@{->}[r] & G \ar@{->}[r]^i & G_S \ar@{->}[r]^{\phi} & S  
\ar@{->}[r] & 1}
\end{equation}
We will also be interested in another exact sequence,
\begin{equation} \label{e.seq2} \xymatrix{  1 \ar@{->}[r] & (\bbG_m^n) \rtimes S \ar@{->}[r]^{\quad \alpha} & G_S \ar@{->}[r]^{\beta} & \PGL_2 \ar@{->}[r] & 1},
\end{equation}
where map $G \to \PGL_2$ sends $(g_1, \dots, g_n) \in G$ to $\overline{g_1} = \dots = \overline{g_n}$.

Consider the natural linear action of $G_S$ on the $2n$-dimensional affine space $V = (\bbA^2)^n$ defined as follows:
$(g_1, \dots, g_n) \in G$ acts on $(\bbA^2)^n$ by
\[ (g_1, \dots, g_n) \, \colon \, (v_1, \dots, v_n) \mapsto (g_1 v_1, \dots, g_n v_n) \]
and $\sigma \in S \subset \Sym_n$ by
\[ \sigma \, \colon \, (v_1, \dots, v_n) \mapsto (v_{\sigma(1)}, \dots, v_{\sigma(n)}) . \]
One readily checks that this action is generically free. (Recall that our standing assumption is that $n \geqslant 5$.)
That is, $V$ has a dense $G$-invariant Zariski open subset $V_0$, such that the stabilizer of $v$ in $G_S$ is trivial
for every $v \in V_0$. After passing to a smaller $G_S$-invariant open subset, we may assume that $V_0$ is the total space of
a $G_S$-torsor $T_S \colon V_0 \to Z_S$ for some $k$-variety $Z_S$; see~\cite[Theorem 4.7]{berhuy-favi}. We thus obtain the following diagram:
\[  \xymatrix{
 & V_0 \ar@{->}[d]_{\text{$(\bbG_m^n \rtimes S)$-torsor}} \ar@{->}@/^1.5pc/[dd]^{\text{$T_S$, a $G_S$-torsor}}   \\
 & Y_S  \ar@{->}[d]_{\text{$t_S$, a $\PGL_2$-torsor}} \\
\eta \ar@{->}[r]                               & 
Z_S } \]
where $Y_S = V_0/(\bbG_m^n \rtimes S)$. The function fields $k(Z_S)$ and $k(Y_S)$ are 
naturally isomorphic to $L_n^S$ and $K_n^S$, respectively. 
When we pass to the generic point $\eta$ of $Z_S$, $T_S$ gives rise to a $G_S$-torsor $(V_0)_{\eta} \to \Spec(K_n^{S})$
and $t_S$ to a $\PGL_2$-torsor $(Y_S)_{\eta} \to \Spec(K_n^S)$, respectively.
By abuse of notation we will continue to denote these torsors by $T_S$ and $t_S$. 

Now let $K$ is an arbitrary field. Recall that $G_S$-torsors over $\Spec(K)$ are classified by the Galois cohomology set $H^1(K, G_S)$, and 
$\PGL_2$-torsors are classified by $H^1(K, \PGL_2)$; see~\cite[\S I.5.2]{serre-gc}. 
We will denote the classes of $T_S$ and $t_S$ by $[T_S] \in H^1(K_n^S, G_S)$ and $[t_S] \in H^1(K_n^S, \PGL_2)$, 
respectively. The exact sequences~\eqref{e.seq1} and~\eqref{e.seq2} of algebraic groups 
give rise to exact sequences of Galois cohomology sets
\begin{equation} \label{e.long1}
\xymatrix{H^1(K, G) \ar@{->}[r]^{i_1} & H^1(K, G_S) \ar@{->}[r]^{\phi_1} &
H^1(K, S) } 
\end{equation}
and
\begin{equation} \label{e.long2}
\xymatrix{H^1(K, \bbG_m^n \rtimes S) \ar@{->}[r]^{\; \; \; \alpha_1} & H^1(K, G_S) \ar@{->}[r]^{\beta_1 \; \;} &
H^1(K, \PGL_2)  } 
\end{equation} 
for any field $K$. If $K = K_n^S$, then by our construction $[t_S] = \beta_1([T_S])$. The following proposition recasts Question~\ref{q.main} in the language
of Galois cohomology.

\begin{proposition} \label{prop.gc} The following conditions on a subgroup $S \subset \Sym_n$ are equivalent:

\smallskip
(a) $L_n^S$ is rational over $K_n^S$,

\smallskip
(b) $L_n^S$ is unirational over $K_n^S$,

\smallskip
(c) $[t_S]$ is the trivial class in $H^1(K_n^S, \PGL_2)$, 

\smallskip
(d) $\beta_1 \colon H^1(K, G_S) \to H^1(K, \PGL_2)$ is the trivial for every field $K$ containing $k$,

\smallskip
(e) $\alpha_1 \colon H^1(K, \bbG_m^n \rtimes S) \to H^1(K, G_S)$ is surjective for every field $K$ containing $k$,

\smallskip
(f) $\phi_1 \colon H^1(K, G_S) \to H^1(K, S)$ is bijective for every field $K$ containing $k$,

\smallskip
(g) $H^1(K, {\, } _{\tau} G) = 1$ for every $\tau \in H^1(K, S)$.
\end{proposition}

In part (g), $_\tau G$ denotes the twist of $G$ by $\tau$ via the natural permutation action of $S$ on $G$.
For generalities on the twisting operation, see~\cite[Section I.5.3]{serre-gc} or~\cite[Section II.5]{berhuy}. 
Note in particular that $_\tau G$ is an algebraic group over $K$; it does not descend to $k$ in general.

\begin{remark} \label{rem.quaternion} The Galois cohomology set $H^1(K, \PGL_2)$ is in a natural
(i.e., functorial in $K$) bijective correspondence with the set of isomorphism classes of
quaternion algebras over $K$; see~\cite[\S I.2 and I.3]{serre03}. Thus condition (c) amounts to saying 
that a certain quaternion algebra over $K_n^S$ is split.
\end{remark}

We defer the proof of Proposition~\ref{prop.gc} to Section~\ref{sect.proof-gc}.

\section{Generalities on Galois cohomology}
\label{sect.generalities}

Suppose $i \colon A \to B$ is a morphism of algebraic groups over $k$, and 
$K$ is a field containing $k$. Following the notational conventions of the previous section,
we will denote the induced map $H^d(K, A) \to H^d(K, B)$ of cohomology sets by $i_d$. Here $d = 0$ or $1$.

The following lemma will be used in the proof of Proposition~\ref{prop.gc}.

\begin{lemma} \label{lem.prel-gc} Consider the exact sequence
\begin{equation} \label{e.seq3} \xymatrix{1 \ar@{->}[r] & A \ar@{->}[r]^{i \;} & B \ar@{->}[r]^{\; \pi} & C  \ar@{->}@/^1pc/[l]^{s} \ar@{->}[r] & 1}
\end{equation}
of smooth algebraic groups over a field $k$. Then

\smallskip
(a) The map $\pi_1 \colon H^1(K, B) \to H^1(K, C)$ is surjective for every field $K/k$.

\smallskip
(b) $\pi_1$ is injective if and only if $H^1(K, {\, }_{s_1(\gamma)} A) = 1$ for every $\gamma \in H^1(K, C)$.
\end{lemma}

\begin{proof} (a) is clear, since $s_1 \colon H^1(K, C) \to H^1(K, B)$ is a section for $\pi_1$.
To prove (b), twist the exact sequence~\eqref{e.seq3} 
by $\tau = s_1(\gamma)$ to obtain a new exact sequence
\[ \xymatrix{  1 \ar@{->}[r] & { \, }_{\tau} A \ar@{->}[r] & _{\tau} B \ar@{->}[r]^{_{\tau} \pi} & {\, }_{\tau} C  \ar@{->}@/^1pc/[l]^{_\tau s} \ar@{->}[r]
&  1 } \]
of algebraic groups over $K$ and consider the associated long exact sequence
\begin{equation} \label{e.las}
\xymatrix{  H^0(K, {\, }_{\tau} B) \ar@{->}[r]^{{_\tau \pi}_0} & 
H^0(K, {\, }_{\tau} C) \ar@{->}@/^1pc/[l]^{{_\tau s}_0} \ar@{->}[r]^{\delta} & H^1(K, {\, }_{\tau} A) 
\ar@{->}[r]^{{_\tau i}_1} & H^1(K, {\, }_{\tau} B) \ar@{->}[r]^{{_\tau \pi}_1} &
H^1(K, {\, }_{\tau}  C) \ar@{->}@/^1pc/[l]^{{_\tau s}_1} } 
\end{equation}
in cohomology. Note that $_{\tau} C$ is naturally isomorphic to $_{\gamma} C$.
By~\cite[Corollary I.5.5.2]{serre-gc}, the fiber of $\gamma$ under $\pi_1$ is in bijective correspondence with the set of $H^0(K, {\, }_{\tau} C)$-orbits in
$H^1(K, {\, }_{\tau} A)$. We are interested in the case, where $\pi_1$ is injective, i.e., this fiber is trivial for every $\gamma \in H^1(K, C)$.

Since ${_\tau s}_0$ is a section for ${_\tau \pi}_0$, we see that ${_\tau \pi}_0$ is surjective. Thus the connecting map $\delta$ in
the long exact sequence~\eqref{e.las}  
sends every element of $H^0(K, {\, }_{\tau} C)$ to the trivial element of $H^1(K, {\, }_{\tau} A)$. Consequently, $H^0(K, {\, }_{\tau} C)$ acts trivially on 
$H^1(K, {\, }_{\tau} A)$. 
We conclude that the fiber of $\gamma$ under $\pi_1$ is in bijective correspondence with $H^1(K, {\, }_{\tau} A)$. In particular,
$\pi_1$ is injective if and only if $H^1(K, {\, }_{\tau} A) = 1$ for every $\gamma$, as claimed.
\end{proof}

\begin{corollary} \label{cor.prel-gc} 
Consider the split exact sequence
\begin{equation} 
\label{e.seq4} 
\xymatrix{  1 \ar@{->}[r] & \bbG_m^n \ar@{->}[r]^{i \; \; \; \; } & \bbG_m^n \rtimes S \ar@{->}[r]^{\; \; \; \; \pi} & S \ar@{->}[r] \ar@{->}@/^1pc/[l]^{s} \ar@{->}[r] & 1}, 
\end{equation}
where $S$ is a subgroup of $\Sym_n$, and $\bbG_m^n \rtimes S$ is the semidirect product with respect to the natural
(permutation) action of $S$ on $\bbG_m^n$. Then $\pi$ and $s$ induce mutually inverse bijections
$\pi_1 \colon H^1(K, \bbG_m^n \rtimes S) \to H^1(K, S)$ and $s_1 \colon H^1(K, S) \to H^1(K, \bbG_m \rtimes S)$ for every
field $K$ containing $k$.
\end{corollary}

\begin{proof} By Lemma~\ref{lem.prel-gc} it suffices to show that
\begin{equation} \label{e.quasi-trivial}
\text{$H^1(K, {\, }_\gamma (\bbG_m^n)) = 1$ for every $\gamma \in H^1(K, S)$.}
\end{equation}
Here $S$ acts on $\bbG_m^n$ by permuting the $n$ copies of $\bbG_m$. Thus the twisted group
${\, }_\gamma (\bbG_m^n)$ is a quasi-trivial torus,
and~\eqref{e.quasi-trivial} follows from the Feddeev-Shapiro Lemma~\cite[Section I.2.5]{serre-gc}.
\end{proof}

\section{Proof of Proposition~\ref{prop.gc}}
\label{sect.proof-gc}

The implication (a) $\Longrightarrow$ (b) is obvious.

\smallskip
(b) $\Longrightarrow$ (c): If $L_n^S$ is unirational over $K_n^S$, then
$t_S$ has a rational section, and (c) follows.

\smallskip
(c) $\Longrightarrow$ (a): If $[t_S] =1 $ is the trivial class in $H^1(K_n^S, \PGL_2)$, then
$Y_S$ is birationally isomorphic to $Z_S \times \PGL_2$ over $Z_S$. Since the group variety of $\PGL_2$
is rational over $k$, this tells us that $Y_S$ is rational over $Z_S$. Equivalently,
$k(Y_S) = L_n^S$ is rational over $k(Z_S) = K_n^S$.

\smallskip
(c) $\Longrightarrow$ (d):
By \cite[Example I.5.4]{serre03}, $[T_S]$ is a versal $G_S$-torsor. This implies that if $[t_S] = \beta_1([T_S])$
is trivial in $H^1(K_n^S, \PGL_2)$, then the image of every element of $H^1(K, G_S)$ under $\beta_1$
is trivial in $H^1(K, \PGL_2)$ for every infinite field $K$ containing $k$, as desired. It remains to consider the case, where
$K$ is a finite field. By Wedderburn's ``little theorem" every quaternion algebra over a finite field $K$ is split. In view of Remark~\ref{rem.quaternion}, 
this translates to $H^2(K, \PGL_2) = \{ 1 \}$. We conclude that the map $H^1(K, G_S) \to H^1(K, \PGL_2)$ is trivial for every field $K$ 
containing $k$.

\smallskip
(d) $\Longrightarrow$ (c) is obvious.

\smallskip
(d) $\Longleftrightarrow$ (e): Follows from the fact that the sequence~\eqref{e.long2} is exact.

\smallskip
(e) $\Longleftrightarrow$ (f): Consider the group homomorphisms $ \xymatrix{S \ar@{->}[r]^{s \quad \; } & \bbG_m^n \rtimes S \ar@{->}[r]^{\; \; \; \; \alpha} & G_S \ar@{->}[r]^{\phi} &  S}$, 
whose composition is the identity map
$S \to S$. Note that here $\phi$, $\alpha$ and $s$ are the same as in~\eqref{e.seq1},~\eqref{e.seq2}, and \eqref{e.seq4}, respectively.
Let us examine the induced sequence 
\[  \xymatrix{  H^1(K, S) \ar@{->}[r]^{\simeq \quad \quad}_{s_1 \quad} \ar@{->}@/_2pc/[rrr]^{\id} &  H^1(K, \bbG_m^n \rtimes S) \ar@{->}[r]^{\quad \alpha_1} &
H^1(K, G_S) \ar@{->}[r]^{\phi_1} &  H^1(K, S)} \]
in cohomology. By Corollary~\ref{cor.prel-gc}, $s_1$ is an isomorphism. Thus 
$\alpha_1$ is surjective if and only of $\phi_1$ is bijective, as claimed.

\smallskip
(f) $\Longleftrightarrow$ (g): Immediate from Lemma~\ref{lem.prel-gc}, applied to the exact sequence~\eqref{e.seq1}.
\qed

\section{Reduction to the case, where $S$ is a $2$-group}

\begin{lemma} \label{lem.prel2} Let $P$ be a subgroup of $S$. Assume that the index $d = [S:P]$ is odd. Then 
$[t_S]$ is trivial in $H^1(K_n^S, \PGL_2)$ if and only if $[t_P]$ is trivial in $H^1(K_n^P, \PGL_2)$.
\end{lemma}

\begin{proof} The diagram
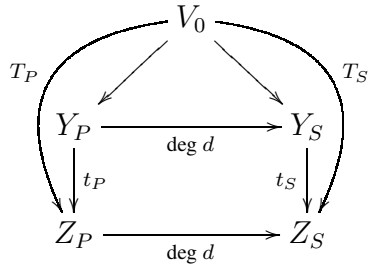
\[  \xymatrix{  &  V_0 \ar@{->}[dl] \ar@{->}[dr] \ar@{->}@/_3pc/[ddl]_{T_P}  \ar@{->}@/^3pc/[ddr]^{T_S}  &  \\
Y_P    \ar@{->}[d]^{t_P}  \ar@{->}[rr]_{\text{deg $d$}}                                                        
&  & Y_S  \ar@{->}[d]_{t_S} \\
Z_P  \ar@{->}[rr]_{\text{deg $d$}}                              &       &   Z_S  } \]
shows that $[t_P]$ is the image of $[t_S]$ under the restriction map
\[ r: H^1(K_n^S, \PGL_2) \to H^1(K_n^P, \PGL_2). \]
By Proposition~\ref{prop.gc} it suffices to show
that $r$ has trivial kernel.
By Remark~\ref{rem.quaternion} elements of the Galois cohomology set $H^1(K, \PGL_2)$ can be 
identified with quaternion algebras over $K$ (up to $K$-isomorphism). The map $r$ sends a quaternion 
algebra $A$ over $K_n^S$ to the quaternion algebra $A \otimes_{K_n^S} K_n^P$ over $K_n^P$. 
Since $K_n^{P}/K_n^S$ is a field extension of odd degree, 
$A \otimes_{K_n^S} K_n^P$ is split if and only if $A$ is split. Thus $r$ has trivial kernel, as claimed.
\end{proof}

Combining Lemma~\ref{lem.prel2} with the equivalence of (a), (b), (c) in Proposition~\ref{prop.gc},
we see that for the purpose of proving Theorem~\ref{thm.main}, 
$S$ may be replaced its $2$-Sylow subgroup $P$. Note that $S$ has an orbit of odd order in $\{ 1, 2, \dots, n \}$ if
and only if $P$ has an orbit of odd order in $\{1, 2, \dots, n \}$ if and only if $P$ has a fixed point.
By the equivalence of parts (a), (b) and (g) in Proposition~\ref{prop.gc}, 
in order to complete the proof of Theorem~\ref{thm.main}, it suffices to establish the following.

\begin{proposition} \label{prop2}  Let $S$ be a $2$-subgroup of $\Sym_n$. Then the following conditions are equivalent.

\smallskip
{\rm (i)} $H^1(K, {\, } _{\tau} G) = 1$ for every $\tau \in H^1(K, S)$.

\smallskip
{\rm (ii)} $S$ has a fixed point in $\{ 1, \dots, n \}$.
\end{proposition}

\section{Conclusion of the proof of Theorem~\ref{thm.main}}

In this section we will complete the proof of Theorem~\ref{thm.main} by establishing Proposition~\ref{prop2}.

Denote the orbits of $S$ in $\{ 1, \dots, n \}$ by $\cO_1, \dots, \cO_t$ where $\cO_i \simeq S/S_i$ as a $G$-set.
Here $S_i$ is the stabilizer of a point in $\cO_i$. The groups $S_1, \dots, S_t$ are uniquely determined by the embedding 
$S \hookrightarrow \Sym_n$ up to conjugacy and reordering. Note that $S_1, \dots, S_t$ may not be distinct.

Recall that elements of $\tau \in H^1(K, S)$ are in a natural bijective correspondence with $S$-Galois algebras $L/K$.
Here by an $S$-Galois algebra $L/K$ we mean an \'etale algebra (i.e., a direct sum of finite separable field extensions of $K$)
equipped with a faithful action of $S$ such that $\dim_K(L) = |S|$ and $L^S = K$; see~\cite[Example 2.2]{serre03}. 
To an $S$-Galois algebra $L/K$ one can naturally associate the \'etale $K$-algebra 
\[ E = L^{S_1} \times \dots \times L^{S_t} \]
of degree $n$.

Now observe that the group $G$ (defined at the beginning of Section~\ref{sect.recast}) admits the following alternative description.
Consider the natural surjective map $f \colon \GL_2 \times \bbG^m \to G$ given by 
\[ (g, t_1, \dots, t_n) \to (gt_1, gt_2, \dots, g t_n) . \]
The kernel of $f$ is $\Delta = \{ (t I_2, t, \ldots, t) \, | t \in \bbG_m \} \simeq \bbG_m$. Thus $f$ induces an isomorphism 
$G \simeq (\GL_2 \times \bbG_m^n)/\Delta$ of algebraic groups. Moreover, this
isomorphism is $S$-equivariant with respect to the natural actions of $S$ on $G$ (described at the beginning of section~\ref{sect.recast})
and $(\GL_2 \times \bbG_m^n)/\Delta$ (via permuting the $n$ components of $\bbG_m$).
The twisted forms $_{\tau} G$ of $G$ and the Galois cohomology sets $H^1(K, {\, }_{\tau} G)$  are explicitly described in~\cite{fr}. In particular,
if $\tau \in H^1(K, S)$ corresponds to the $n$-dimensional \'etale $K$-algebra $E$ as above,
then ${\, }_{\tau} G \simeq (\GL_2 \times R_{E/K}(\bbG_m))/\Delta_K$, where $R_{E/K}$ denotes Weil restriction and $\Delta_K \simeq {\, }_{\tau} \Delta$
is $\bbG_m$ (over $K$), diagonally embedded into $\GL_2 \times R_{E/K}(\bbG_m)$; see~\cite[Section 4]{fr}.
Moreover,
\[ H^1(K, {\, }_{\tau} G) \simeq \; \{ \text{isom.~classes of quaternion $K$-algebras $A$ such that $A$ is split by $E \otimes_k K \}$;} 
\]
see~\cite[Lemma 5.1]{fr}. Note that $A$ is split by $E \otimes_k K$ if and only if $A$ is split by $L^{S_i} \otimes_k K$ for every $i = 1, \dots, t$.
This explicit description of $H^1(K, {\, }_{\tau} G)$ reduces Proposition~\ref{prop2} to the following equivalent form.

\begin{proposition} \label{prop3} 
Let $S \subset \Sym_n$ be a $2$-group. Then the following conditions on $S$ are equivalent.

\smallskip
(a) There exists a field a field extension $K/k$, a quaternion division algebra $A/K$ and $S$-Galois algebra $L/K$ such that
$A$ splits over $L^{S_i}$ for every $i = 1, \dots, t$.

\smallskip
(b) $S$ does not have a fixed point in $\{ 1, \dots, n \}$ .
\end{proposition}

To prove (a) $\Longrightarrow$ (b), assume that $S$ has a fixed point. That is, one of the orbits of $S$ in $\{ 1, \dots, n \}$, 
say $\cO_1$, consists of a single point. Equivalently, $S_1 = S$.  Clearly a quaternion division algebra 
over $K$, cannot split over $L^{S_1} = L^S = K$.

To prove (b) $\Longrightarrow$ (a), assume $S$ does not have a fixed point in $\{ 1, 2, \dots, n \}$, i.e.,
$|\cO|_i = |S/S_i| \geq 2$ for every $i = 1, \dots, t$. 
Since $S$ is a $2$-group, each $S_i$ is contained in a maximal
proper subgroup of $S$. That is, for each $i = 1, \dots, t$, there exists a subgroup
\[ \text{$S_i \subseteq H_i \subsetneq S$ such that $[S: H_i] = 2$.} \]
Note that $H_i$, being a subgroup of index $2$, is normal in $S$.

Now let $M/F$ be an $S$-Galois field extension. For example, we can let $S$ act on $M = k(x_1, \dots, x_n)$ by permuting the variables
and set $F = M^S$. Since $H_i$ is normal in $S$, $M^{H_i}/F$ is a Galois extension of degree $2$ for each $i = 1, \dots, t$. 
By a theorem of M.~Van den Bergh\ and\ A.~Schofield~\cite[Theorem 3.8]{vdbs}, there exists a field extension $K/F$ and a quaternion 
division algebra $A/K$ such that $A$ contains $M^{H_i} \otimes_{F} K$ as a maximal subfield for each $i = 1, \dots t$.
Now consider the $S$-Galois algebra $L = M \otimes_F K$ over $K$. For each $i = 1, \dots, t$, $L^{S_i}$ contains $L^{H_i} = M^{H_i} \otimes_F K$.
Hence, each $L^{S_i}$, splits $A$, as desired. 

This completes the proof of Proposition~\ref{prop3} and thus of Proposition~\ref{prop2} and Theorem~\ref{thm.main}. 
\qed

\section*{Acknowledgments} The author is grateful to Skip Garibaldi for helpful comments.

\end{document}